\documentclass{elsarticle}
\usepackage{lineno,hyperref}
\usepackage{amsmath}

\modulolinenumbers[5]

\journal{Applied Mathematics and Computation}

\usepackage{amssymb}
\usepackage{microtype}
\newtheorem{thm}{Theorem}[section]

\newtheorem{prop}[thm]{Proposition}
\newproof{proof}{Proof}

\newcommand{\be}{\begin{equation}}
\newcommand{\ee}{\end{equation}}

\renewcommand{\H}{\mathcal{H}}
\newcommand{\R}{\mathbb R}
\newcommand{\A}{\mathcal A}
\newcommand{\C}{\mathbb C}
\newcommand{\No}{\mathbb{N}\cup\{0\}}
\newcommand{\Z}{\mathbb Z}
\renewcommand{\phi}{\varphi}
\newcommand{\set}[1]{\left\{#1\right\}}
\renewcommand{\sc}[2]{\langle #1|#2 \rangle}
\newcommand{\norm}[1]{\left\Vert#1\right\Vert}

\renewcommand{\t}{\tau}

\newcommand{\dt}{\partial_\tau}
\newcommand{\T}{\mathbb T}

\numberwithin{equation}{section}

\begin{document}
\begin{frontmatter}

\title{Factorization method on time scales}

\author{Tomasz Goli\'nski}
\ead{tomaszg@math.uwb.edu.pl}
\address{\small University of Bia{\l}ystok, Institute of Mathematics\\Cio{\l}kowskiego 1M, 15-245~Bia{\l}ystok, Poland}

\begin{abstract}
  We present an approach to the factorization method for second order difference equations on time scales. We construct Hilbert spaces of functions on the time scale and show how to construct a chain of intertwined first order $\Delta$-difference operators $\A_k$ 
 in the sense that $\A_k$ and $\A_k^*$ act as ladder operators generating new solutions of eigenproblem of $\A_k\A_k^*$.
\end{abstract}


\end{frontmatter}


\section{Introduction}

The factorization method is one of the most useful tools for solving second order differential and difference equations, including Schr\"odinger equation. It dates back to the works by Darboux \cite{darboux} and it was successfully applied to many problems especially in quantum mechanics, see e.g. \cite{schrodinger-factorization,infeld,mielnik1984factorization}. A modern review of the method can be found in \cite{mielnik2004factorization,dong2007factorization}. It also has broader applications in mathematics including theories of integrable systems, orthogonal polynomials and special functions, see e.g. \cite{hahn,teschl}. Other applications of the factorization method in discrete case can be found e.g. in \cite{lorente,alvarez}.

The aim of this paper is to reformulate and generalize the results of papers \cite{OG,GO} for the case of calculus on time scales introduced in \cite{hilger1990}. Namely in those papers a factorization method for second order difference equations of the form
\be \label{0}\alpha(x)\psi(\t^2(x))+\beta(x)\psi(\t(x))+\gamma(x)\psi(x)= 0 \ee
was investigated, where $\t:X\to X$ is a bijection of a certain subset $X\subset \R$ of the real line. By introduction of $\t$-difference operator $\dt$ (see also e.g. \cite{klimek,ruffing2003difference})
\be\label{dt}\dt\phi(x):=\frac{\phi(x)-\phi(\t(x))}{x-\t(x)} \ee
equation \eqref{0} can be rewritten as 
\be \label{1}\tilde\alpha(x)\dt^2\psi(x)+\tilde\beta(x)\dt\psi(x)+\tilde\gamma(x)\psi(x)=0.\ee
The approach is based on the construction of a chain of Hilbert spaces $\H_k$ and representing left hand side of \eqref{0} as $\A^*_k \A_k \psi$ for some first order difference operators $\A_k:\H_k\to \H_{k+1}$ satisfying the relation
\be \A_k\A_k^*=a_k \A_{k+1}^*\A_{k+1} + b_k.\ee
This allows us to obtain new eigenfunctions of $\A_k^*\A_k$ from any known eigenfunction. It means that $\A_k$ and $\A_k^*$ behave as annihilation and creation operators (or ladder operators) of quantum mechanics. It can be also seen as an example of a ladder in the sense of \cite{hilger2010ladders}.
This approach to the factorization method for the equations \eqref{0} led to many results in the particular cases of difference, $q$-difference, and $(q,h)$-difference equations and they can be found e.g. in \cite{Alanowa,alina-qschroed,DGO,alina-filipuk,alina-grzegorz-q,alina-hounkonnou,alina-grzegorz-h}. 

In the case when the map $\t$ is strictly increasing and $X$ can be generated from a single point the operator $\dt$ is exactly the time scales $\Delta$-derivative and $\t$ plays role of forward shift operator $\sigma$. 
On one hand the time scales calculus can be seen as a generalization of $\dt$-calculus since it covers also differential case. On the other hand, the freedom of choosing $\t$ gives more possibilities than discrete component in the time scales approach --- for example $\tau$ can be chosen to be cyclic $\tau^N=\textrm{id}$. Our belief is that reformulation of results from aforementioned papers in another, more popular language might be valuable for further applications. 
However for the purpose of this paper we will restrict our attention to time scales where every point is isolated (excluding $\max \T$ and $\min \T$ if they exist, which are assumed to be one-side scattered). Treatment of the most general case poses more technical difficulties and will be investigated in the future.

In this paper we will not describe in details the $\dt$-calculus but concentrate solely on time scales case. In Section 2 we will briefly introduce elements of the time scales calculus. For more extensive presentations of the subject, see e.g. \cite{agarwal2002dynamic,bohner2002advances}. We also introduce the Hilbert spaces and operators which will be used later. Section 3 contains the main result, i.e. a method to construct a chain of factorized second order operators as well as an example in the case $\T=\Z$.

\section{Hilbert spaces and operators}

Let $\T$ be a closed subset of $\R$. 
One denotes by $\sigma$ a shift (forward jump) acting on time scale $\T\to\T$ defined by
\be \sigma(x) := \inf \;\{ t\in\T\; | \; t>x \}\ee
(where one additionally puts $\inf \emptyset := \max \T$) and the graininess of $\T$ is the following function 
\be\mu(x):=\sigma(x)-x.\ee 
For any set $X$, we denote by $X^\kappa$ (resp. $X_\kappa$) the set $X$ with the maximal point (resp. minimal point) removed.
If there is no maximal (minimal) point in $X$, these symbols denote just $X$.

For the purpose of this paper we restrict our attention to the sets 
where every point is isolated (excluding $\max \T$ and $\min \T$ if they exist, which are assumed to be only one-side scattered), i.e. $\sigma(x)\neq x$ for $x\in \T^\kappa$. The $\Delta$-derivative in that setting can be defined on $\T^\kappa$ as 
\be f^\Delta:=(f\circ\sigma-f)/\mu.\ee

There exists the measure $\mu_\Delta$ on $\T$ realizing the $\Delta$-integral, called Lebesgue $\Delta$-measure, see \cite{guseinov2003integration}.
We will consider the Hilbert space $L^2(\T,\rho \mu_\Delta)$ of (classes of equivalence of) square integrable functions on the time scale $\T$
with a fixed positive-valued weight function $\rho:\T^\kappa\to\R_+\cup\set0$.
The scalar product is given by
\be \label{sc}\sc{\psi}{\phi}:=\int_\T \overline{\psi(x)}\phi(x)\rho(x)\mu_\Delta(x)\ee
for functions $\psi,\phi:\T\to\C$. Note that some basic facts about these spaces can be found in \cite{ruffing2005corresponding}.

Note that since we assume that all points of $\T^\kappa_\kappa$ are isolated, the measure $\mu_\Delta$ is a pure point measure.
One usually assumes that the set $\set{\max \T}$ (if it exists) has infinite measure. However that breaks the expected behaviour of $\Delta$-Lebesgue integral compared to $\Delta$-Riemann integral for functions which don't vanish there. We will follow a more convenient convention and put $\mu_\Delta(\set{\max \T})=0$. In this way we get
\be \label{measure-point}\mu_\Delta(\set x) = \mu(x)\ee
for all $x\in \T$. Regardless of this choice, the $L^2$ space considered here is essentially the same and will be identified with set of functions on $\T^\kappa$ with
the scalar product expressed as a sum
\be \sc{\psi}{\phi}=\sum_{x\in \T^\kappa} \overline{\psi(x)}\phi(x)\rho(x)\mu(x).\ee
If $\T$ is unbounded then this space is infinite dimensional. 

In the paper we are going to use shift operators acting on functions. In order to simplify the formulas we introduce the notation
\be \label{shift} f^\sigma(x) := f(\sigma(x)) \; \chi_{\T^{\kappa\kappa}}(x)\ee
and
\be \label{shift-1} f^{\sigma^{-1}}(x) := f(\sigma^{-1}(x)) \; \chi_{\T^\kappa_\kappa}(x),\ee
where $\chi_A$ is the characteristic function of the set $A$. Terms $\chi_{\T^{\kappa\kappa}}$ and $\chi_{\T^\kappa_\kappa}$ are included in order not to duplicate values in some points. It is an analogue to the way shift operators are defined in the Toeplitz algebra.

Consider the shift operator $S$ acting in $L^2(\T,\rho \mu_\Delta)$ defined as
\be \label{S-def} S\psi := \psi^\sigma.\ee
It is in general unbounded and thus it is only defined on a dense domain. This domain can be chosen to contain functions with finite support (which are also dense in $L^2(\T,\rho \mu_\Delta)$). We will omit the detailed discussion of the domains and treat all operators formally. 
In case when $\T$ is bounded from above or below, the operator $S$ is not invertible. 
The adjoint operator is defined by
\be \sc{\psi}{S\phi}=\sc{S^*\psi}{\phi}\ee
and straightforward calculation produces the formula
\be \label{S*} S^*\psi=\frac{\rho^{\sigma^{-1}}}{\rho}\; (\sigma^{-1})^\Delta \; \psi^{\sigma^{-1}}.\ee
It is useful to note that composition of $S$ and $S^*$ is an operator of multiplication by a function:
\be \label{SS} S^*S =  \frac{\rho^{\sigma^{-1}}}{\rho}\; (\sigma^{-1})^\Delta ,\qquad 
SS^* = \frac{\rho}{\rho^\sigma}\; (\sigma^{-1})^{\Delta,\sigma}. \ee 
As a consequence one finds that the norm of the operator $S$ is equal
\be \norm S = \sqrt{\norm{S^*S}} = \sqrt{\sup_{x\in \T^\kappa_\kappa} \frac {\rho(\sigma^{-1}(x))}{\rho(x)}\;(\sigma^{-1})^\Delta(x)} \ee
and if the supremum is finite then the operator is bounded and can be defined on the whole $L^2(\T,\rho \mu_\Delta)$.

Let us now consider a sequence of weight functions $\rho_k:\T^\kappa\to\R_+\cup\set0$ and corresponding Hilbert spaces $\H_k:=L^2(\T,\rho_k \mu_\Delta)$ for $k$ in some subset of $\Z$. Following \cite{OG} we assume that weight functions satisfy a recurrence relation
\be \rho_{k+1}=(B_k\rho_k)^\sigma\ee
and the following Pearson-like equation
\be \label{pearson}(B_k \rho_k)^\Delta=A_k\rho_k\ee
for some functions $A_k, B_k:\T\to\R$. It is an analogue of the situation in the theory of classical ($q$-)orthogonal polynomials where
$A_k$, $B_k$ are polynomials of respectively first and second order and orthogonality measure is a solution of a similar equation.

By introducing a new function family
\be \label{eta_def}\eta_k := B_k+\mu A_k\ee
and substituting it in place of $A_k$, equation \eqref{pearson} assumes the form
\be \label{eta_rho}\rho_{k+1}=\eta_k\rho_k.\ee

Let us now define further operators acting on the chain of Hilbert spaces $\H_k$.
In general these operators are also unbounded.
First, consider an operator of inclusion $I_k:\H_k\to\H_{k+1}$ acting as
\be I_k\psi := \psi.\ee
The adjoint operator 
can be computed using \eqref{sc}:
\be \sc{\psi}{I_k\phi}_{k+1}=\int_T \overline{\psi(x)} \phi(x)\rho_{k+1}(x) \mu_\Delta(x) = 
\int_T \overline {\psi(x)} \phi(x)\eta_k(x)\rho_{k}(x) \mu_\Delta(x).\ee
Thus we conclude
\be \label{I*} I_k^*\psi = \eta_k \;\psi.\ee


Since $\Delta$-derivative of functions defined on $\T^\kappa$ is defined only on $\T^{\kappa\kappa}$, in order to define a Hilbert space operator representing it, we will follow the convention used in definition of the operator $S$, namely we put $\Delta:\H_k\to\H_{k+1}$:
\be \Delta\psi(x) := \psi^\Delta(x) \chi_{\T^{\kappa\kappa}}(x).\ee
This definition is equivalent to 
\be \Delta = \frac{I_k}{\mu} (S-\chi_{\T^{\kappa\kappa}}).\ee
In order to compute adjoint of $\Delta$
\be \sc{\psi}{\Delta\phi}_{k+1}=\sc{\Delta^*\psi}{\phi}_{k}\ee
we use formulas \eqref{S*} and \eqref{I*} and obtain
\be \Delta^* \psi = (S^*-\chi_{\T^{\kappa\kappa}})\frac{\eta_k}\mu \psi.\ee
Alternatively one could derive this formula using integration by parts (see \cite[Thm. 1.28]{bohner2002advances}).

\section{Factorization method}\label{sec:factor}


We consider a sequence (chain) of the factorized second-order equations
\be \label{t-rown} \A_k^*\A_k\psi = \lambda\psi \ee
for $\psi\in\H_k$ with the operators $\A_k:\H_k\to \H_{k+1}$ given as first order $\Delta$-difference operators
\be\label{A}\A_k:=h_k\Delta + f_k I_k,\ee
where $f_k, h_k :\T\to\R$ and $k\in \No$.

Note that we consider $\A_k$ as a closed operator defined by \eqref{A} on a suitable dense domain containing functions with finite support. The adjoint operator $\A_k^*$ is automatically closed. As a consequence by von Neumann theorem (see e.g. \cite{akhiezer}) $\A_k^*\A_k$ is self-adjoint. It follows that the spectrum of $\A_k^*\A_k$ is real and thus $\lambda\in\R$. Another consequence is that eigenvectors for different eigenvalues are pairwise orthogonal.

The essence of the factorization method is to postulate the relation
\be \label{comm} \A_k\A_k^*=a_k \A_{k+1}^*\A_{k+1} + b_k\ee
for some constants $a_k$, $b_k\in \R$, $a_k\neq 0$, and all $k$. It is an analogue of classical commutation relations for creation and annihilation operators for a harmonic oscillator. Extra parameters  $a_k$, $b_k$ are included for greater flexibility.


The formula \eqref{comm} implies that from any solution $\psi\in\H_k$ of equation \eqref{t-rown} for some $k$ with the eigenvalue $\lambda$, one can obtain 
a new solution 
\be \label{new_sol}\psi^\uparrow=\A_k\psi \in\H_{k+1}\ee
for $k+1$ with the eigenvalue $\lambda^\uparrow = \frac{\lambda-b_k}{a_k}$ or conversely
\be \label{new_sol2}\psi^\downarrow=\A^*_{k-1}\psi \in\H_{k-1}\ee
for $k-1$ with the eigenvalue $\lambda^\downarrow = a_{k-1}\lambda+b_{k-1}$.

One of the methods of profiting from this observation is by noticing that the kernels of the operator $\A_k^*\A_k$ and the operator $\A_k$ coincide. It follows from the fact that for (possibly unbounded) operators acting in a Hilbert space one has $\ker \A_k^* \perp \operatorname{ran} \A_k$.
Usually it is much easier to find the elements of the kernel of the operator $\A_k$ than to solve the generic eigenproblem \eqref{t-rown}. So, if we succeed in finding elements $\psi_l$ of kernel of $\A_l$ for all $l\in\No$ we can obtain a sequence of solutions of equation \eqref{t-rown} for any fixed $k$:
\be \label{kern-sol}\psi_l^k = \A^*_{k}\ldots\A^*_{l-1}\psi_l \in\H_k\ee
for $l>k$. One can also perform the similar procedure with elements of the kernels of the operator $\A_k^*$.

In general this set of solutions needs not to be complete, but in many particular cases (e.g. for harmonic oscillator) it turns out to span the whole Hilbert space. 

The tricky part is to obtain the sequence of operators $\A_k$ satisfying the relation \eqref{comm}. To simplify the formulas we change the notation by replacing the function $f_k$ by the function
\be \label{phi_def}\phi_k:=f_k-\frac{h_k\chi_{\T^{\kappa\kappa}}}{\mu}.\ee
In this way we can express operators $\A_k$ and their adjoints $\A_k^*$ as
\be \A_k = I_k (h_k/\mu\; S+\phi_k),
\qquad \A_k^*
= (S^*\; h_k/\mu+\phi_k)I_k^*.\ee
Note that operators $\A_k^*$ are related to backward $\nabla$-derivative.

If we additionally denote by $g_k$ the ratio of subsequent functions $B_k$:
\be \label{g}B_{k+1}=g_k B_k,\ee
we conclude that similar rule is satisfied by functions $\eta_k$:
\be \label{eta}\eta_{k+1}=(\eta_k g_k)^\sigma.\ee

Using this notation we obtain the following proposition:

\begin{prop}
The equation \eqref{comm} is equivalent to the following set of relations
\be \label{phi}h_{k+1}\phi_{k+1}=\frac{1}{a_k} \frac{h_k \; \phi_k{}^\sigma}{g_k{}^\sigma},\ee
\be\label{eq}
a_k\; g_k\left(\frac{B_k\;({h_{k+1}}^{\sigma^{-1}})^2}{\mu\; \mu^{\sigma^{-1}}}-\frac1{a_k}\frac{{\phi_k}^2\eta_k}{g_k}\right)+b_k=\frac{{B_k}^\sigma\; {h_k}^2}{\mu^\sigma\; \mu}-\frac1{a_k}\frac{({\phi_k}^\sigma)^2 \; {\eta_k}^\sigma}{{g_k}^\sigma}\frac{h_k{}^2}{h_{k+1}{}^2}\ee
on the functions $B_k$, $\eta_k$, $\phi_k$, $g_k$, $h_k$, $h_{k+1}$ and constants $a_k$, $b_k$ for all $k$.
\end{prop}

\begin{proof}
The proof is technical and requires expanding right and left hand sides of \eqref{comm} using identities \eqref{SS} and noting that
\be S(f\psi) = f^\sigma S\psi, \qquad S^*(f\psi) = f^{\sigma^{-1}} S^*\psi, \ee
\be I_k S^* I_k^* = \eta_k S^*.\ee
Comparing the coefficients in front of $S$, $S^*$ and the free term one obtains the following equations
\be \frac{h_k \phi_k{}^\sigma \eta_k{}^\sigma}{\mu} = a_k \frac{\phi_{k+1}\eta_{k+1} h_{k+1}}{\mu},\ee
\be \frac{h_k{}^{\sigma^{-1}} \phi_k \eta_k}{\mu{}^{\sigma^{-1}}} = a_k \frac{\phi_{k+1}{}^{\sigma^{-1}}\eta_{k+1}{}^{\sigma^{-1}} h_{k+1}{}^{\sigma^{-1}}}{\mu^{\sigma^{-1}}},\ee
\be \frac{{h_k}^2 B_k{}^\sigma ((\sigma^{-1})^\Delta)^\sigma \chi_{\T^{\kappa\kappa}}}{\mu^2} + {\phi_k}^2\eta_k = 
a_k\left(\frac{({h_{k+1}}^2)^{\sigma^{-1}} B_k g_k (\sigma^{-1})^\Delta \chi_{\T^\kappa_\kappa}}{(\mu^2)^{\sigma^{-1}}} + {\phi_{k+1}}^2\eta_{k+1}\right)+b_k.\ee
Using identities obtained before one shows that the first and second equations are equivalent to \eqref{phi} and the third is equivalent to \eqref{eq}.
\qed
\end{proof}

In the particular case when $h_k \equiv 1$ for all $k$, the equation \eqref{phi} can be seen as a rule to obtain $\phi_{k+1}$ from $\phi_k$ using $g_k$ (just like \eqref{g} and \eqref{eta}). Thus the sequences of functions $B_k$, $\eta_k$, $\phi_k$ can be obtained using \eqref{g}, \eqref{eta}, \eqref{phi} from their first terms $B_0$, $\eta_0$, $\phi_0$ and the sequence of functions $g_k$ which have to satisfy equations
\be
a_k\; g_k\left(\frac{B_k}{\mu\; \mu^{\sigma^{-1}}}-\frac1{a_k}\frac{{\phi_k}^2\eta_k}{g_k}\right)+b_k=
\left(\frac{B_k}{\mu\; \mu^{\sigma^{-1}}}-\frac1{a_k}\frac{{\phi_k}^2\eta_k}{g_k}\right)^\sigma .
\ee

These are non-linear functional equations on functions $g_k$ (closely related to Riccati equation) but still some classes of solutions are known.

For example in the case of $q$-calculus $\T = \overline{c q^{\No}}$, $0<q<1$, $c\in\R$, one can put $h_k\equiv 1$, $f_k\equiv 0$ and require that $A_k$ are polynomials of the first degree and  $B_k$ polynomials of the second degree. Then the operator $\A_k$ is just the $q$-derivative $\partial_q$ and so $\ker \A_k$ consists of constant functions. Thus one obtains by \eqref{kern-sol} polynomials as eigenvectors of $\A_k^*\A_k$. From self-adjointness of $\A_k^*\A_k$ it follows that they are orthogonal with respect to the measure $\rho_k d_q$ on the time scale $\T$. The equations \eqref{phi}-\eqref{eq} in this case become trivial and give consistency conditions on the value of constants $a_k$ and $b_k$, see \cite[Example 5.2]{GO} for details. Solving equation \eqref{pearson} one can find the measure $\rho_k$ given the polynomials $A_k$ and $B_k$. In this way many known families of $q$-orthogonal polynomials can be obtained. Note that this case requires a slight relaxation of the conditions assumed here for the time scale.

Other particular examples are mostly obtained by postulating some forms of structural functions, e.g. $g_k \equiv \operatorname{const}$.
Then solutions can be found by such methods as iteration, infinite matrix products, power series expansions, see aforementioned papers \cite{OG,Alanowa,alina-qschroed} for more details and examples with applications e.g. for a $q$-deformed Schr\"odinger equations.

Since the problem is underdetermined, another approach to solving the equation \eqref{eq} would be to consider it as a non-explicit rule to obtain $h_{k+1}$ from $h_k$ given initial values $h_0$, $B_0$, $\eta_0$, $\phi_0$ and the whole sequence of functions $g_k$. Consider the simplest case with trivial weight functions $\rho_k \equiv 1$. For consistency we need $B_k\equiv 1$, $\eta_k\equiv 1$, $g_k\equiv 1$. In this case equations \eqref{phi}-\eqref{eq} assume the form:
\be h_{k+1}\phi_{k+1}=\frac{1}{a_k} h_k \; \phi_k{}^\sigma,\ee
\be a_k\; \frac{({h_{k+1}}^{\sigma^{-1}})^2}{\mu\; \mu^{\sigma^{-1}}}+a_k\phi_{k+1}^2+b_k=\frac{{h_k}^2}{\mu^\sigma\; \mu}+{\phi_k}^2. \ee
If we take $\T=\Z$, then $\sigma(x)=x+1$ and $\mu(x)=1$.
Let's start from $h_0(x)=x$, $\phi_0(x)=-x$ and try to compute the operators $\A_1$ and $\A_1^*$. We get the equations
\be h_1(x)\phi_1(x) = -\frac{1}{a_0} x (x+1),\ee
\be a_0 (h_1(x-1))^2+a_0 \phi_1(x)^2 + b_0 = 2 x^2.\ee
We can take as a solution $h_1(x)=x+1$, $\phi_1(x)=-x$, $a_0=1$, $b_0=0$ and obtain the operators 
\be \A_0 = x\,\Delta, \qquad \A_0^* = -\nabla\, x,\ee
\be \A_1 = (x+1)\,\Delta + 1, \qquad \A_1^* = -x\, \nabla.\ee
Thus the equation \eqref{t-rown} assumes the following forms for $k=0$ and $k=1$:
\be -x^2\psi(x+1) + (2x^2-2x+1) \psi(x) + (x-1)^2\psi(x) = \lambda \psi(x),\ee
\be -x\big( (x+1)\psi(x+1)-2x\psi(x)+(x-1)\psi(x-1)\big) = \lambda \psi(x).\ee
From solutions of either of them, one can obtain solutions of the other by the ladder property \eqref{new_sol}-\eqref{new_sol2}.

\end{document}